\documentclass[11pt,a4paper]{article}
\usepackage[utf8]{inputenc}
\usepackage[a4paper, total={6in, 8in}]{geometry}
\usepackage{xcolor,amsmath}
\usepackage{hyperref}
\usepackage{amsthm}
\usepackage{amsfonts}
\usepackage{amssymb}
\usepackage{cite}
\usepackage{bbm}
\usepackage{relsize}
\usepackage{blindtext, titlefoot}
\usepackage{sectsty}
\sectionfont{\fontsize{12}{15}\selectfont}
\newtheorem{corollary}{Corollary}
\newtheorem{lemma}{Lemma}
\newtheorem{theorem}{Theorem}

\newtheorem*{corollary*}{Corollary}

\def\e{\,{\rm{e}}}

\title{On the Vertical Distribution of Values of $L$-functions in the Selberg Class}

\author{Athanasios Sourmelidis $\cdot$ Teerapat Srichan $\cdot$ J\"orn Steuding}

\begin{document}

\date{}

\maketitle

\begin{abstract}
\noindent
We prove explicit formulae for $\alpha$-points of $L$-functions from the Selberg class. 
Next we extend a theorem of Littlewood on the vertical distribution of zeros of the Riemann zeta-function $\zeta(s)$ to the case of $\alpha$-points of the aforementioned $L$-functions. 
This result implies the uniform distribution of subsequences of $\alpha$-points and from this a discrete universality theorem in the spirit of Voronin is derived.
\end{abstract}

\section{Introduction and Statement of the Main Results}

The Riemann zeta-function $\zeta(s)$ plays a central role in number theory and, in particular, the distribution of its zeros is relevant for the error term in the prime number theorem and further asymptotic formulae in analytic number theory. In his invited talk at the fifth International Mathematical Congress at Cambridge in 1912, Landau \cite{astart} suggested that ``[g]iven an analytic function, the points for which this function is $0$ are very important; however, of equal interest are those points where the function assumes a given value''. 

Given a complex number $\alpha$, the roots of the equation 
$$
\zeta(s)=\alpha
$$
are called $\alpha$-points and will be denoted by $\rho_{\alpha}=\beta_{\alpha}+i\gamma_{\alpha}$. 
Moreover, for the non-trivial zeros of $\zeta(s)$ we just use the notation $\rho=\beta+i\gamma$.
Our investigations have their origin in the work of Landau \cite{lan} who proved that if $x>1$, then
$$\sum\limits_{0<\gamma\leq T}x^\rho=-\dfrac{T}{2\pi}\Lambda(x)+O_x(\log T)$$
for any $T>1$, where $\Lambda(x)$ is the von Mangoldt function extended to the whole real line by setting it to be equal to zero if $x$ is not a positive integral power of a prime $p$.
The dependence of the implicit constant on $x$ was made explicit only much later by Gonek \cite{Gon}.

Such explicit formulae have been proven very useful to the study of the vertical distribution of the non-trivial zeros of $\zeta(s)$.
For instance, Radamacher \cite{rad} employed Landau's formula and proved that the sequence $(a\gamma)_{\gamma>0}$ is uniformly distributed mod 1 for any real number $a\neq0$, provided that the Riemann Hypothesis is true.
Later Elliott \cite{Ell} remarked that the latter condition can be removed, and (independently) Hlawka \cite{Hl} obtained the aforementioned result unconditionally.
On the other hand,
Ford and Zaharescu \cite{F.Z} and Ford, Soundararajan and Zaharescu \cite{F.S.Z} made use of Gonek's explicit version of Landau's formula to prove quantitative results regarding the distribution of $\left(a\gamma\right)_{\gamma>0}$ with respect to its discrepancy.
For related results we also refer to Fujii \cite{F1,F2,F3,F4,F5,F6,F7}.

Motivated by such considerations, the third author \cite{st} studied the distribution of the $\alpha$-points of $\zeta(s)$.
He first proved the analogue of Landau's formula, that is, if $\alpha\neq1$ is complex number, $x\neq1$ is a positive number and $\epsilon>0$, then
\begin{align}\label{Reh}
\mathop{\sum\limits_{0<\gamma_\alpha\leq T}}_{\beta_\alpha>0}x^{\rho_\alpha}
=\dfrac{T}{2\pi}\left(a(x)-x\Lambda\left(\dfrac{1}{x}\right)\right)+O_{\alpha,x,\epsilon}\left(T^{1/2+\epsilon}\right)
\end{align}
for any $T>1$, where $a(x)$ is defined on the positive integers by the coefficients of the ordinary Dirichlet series
$$\dfrac{\zeta(s)}{\zeta(s)-\alpha}=\sum\limits_{n=1}^{\infty}\dfrac{a(n)}{n^s},$$
and is equal to zero for any other real number $x$.\footnote{The special case of $\alpha=1$ will be discussed in the remark following Theorem 3.}
He then employed a classic result due to Levinson \cite{levia}, regarding the clustering of the $\alpha$-points of $\zeta(s)$ around the critical line, and proved that the sequence $(a\gamma_\alpha)_{\gamma_\alpha>0}$ is uniformly distributed mod 1 for any real number $a\neq0$.
Recently,
Rehberg \cite{Reh} and Gonek and Baluyot \cite{bg} proved (independently) relation \eqref{Reh} and made the dependence on $x$ explicit; in accordance they proved discrepancy estimates for the sequence $(a\gamma_\alpha)_{\gamma_\alpha>0}$.

We improve upon these results in Theorem \ref{MAIN3}.
As a matter of fact, we start by proving analogous results for functions $\mathcal{L}(s)$ in the Selberg class $\mathcal{S}$ of a given degree $d_\mathcal{L}$.
The definition and properties of $\alpha$-points of such functions are similar to the one of $\zeta(s)$, which belongs to $\mathcal{S}$ with $d_\zeta=1$.
For more details we refer to Section 2.

To state our theorems, we only need to define here for a given complex number $\alpha\neq1$ and $\mathcal{L}\in\mathcal{S}$, the function $\Lambda(x;\mathcal{L},\alpha)$ which for positive integers $x=n$ is given by the coefficients of the  Dirichlet series 
\begin{align}\label{coeff}
\dfrac{\mathcal{L}'(s)}{\mathcal{L}(s)-\alpha}=-\sum\limits_{n=1}^{\infty}\dfrac{\Lambda(n;\mathcal{L},\alpha)}{n^s},
\end{align}
and is equal to zero for any other real number $x$. 
The error terms in the subsequent theorems depends on the real quantity $\sigma(\mathcal{L},\alpha)$, which is related to the convergence of the latter Dirichlet series and is explicitly given by (\ref{aabscissa}).
Laslty, the {\it Landau symbols} $o(\cdot)$ and $O(\cdot)$  have their usual meaning, and if the implied constant depends on some parameter $\epsilon$ (say), then we write $O_\epsilon$.
The implicit constants may depend also on $\mathcal{L}$ and $\alpha$ but we drop such subscripts.
The same comments apply to the {\it Vinogradov symbols} $\ll$ and $\gg$.

\begin{theorem}\label{MAIN1}
Let $\mathcal{L}\in \mathcal{S}\setminus\lbrace1\rbrace$, $\alpha\neq1$ be a complex number and $\epsilon>0$. Then
\begin{align*}
\mathop{\sum\limits_{0<\gamma_\alpha\leq T}}_{\beta_\alpha\geq0}x^{\rho_\alpha}
=&-\dfrac{T}{2\pi}\Lambda\left(x;\mathcal{L},\alpha\right)+O_\epsilon\left(x^{\sigma(\mathcal{L},\alpha)+\epsilon}\left(1+\mathbbm{1}_{\mathbb{R}\setminus\mathbb{Z}}(x)\min\left\{\frac{T}{x},\dfrac{1}{\|x\|}\right\}\right)\right)+\\
&+O_\epsilon\left(x^{\sigma\left(\mathcal{L},\alpha\right)+\epsilon}\log T\left(1+\dfrac{1}{\log x}\right)\right)+\\
&+O\left(\frac{1}{x^2}\left(\log T\left(1+\dfrac{1}{\log x}\right)+\min\left\{T,\dfrac{1}{\log x}\right\}\right)\right)
\end{align*}
for any $x,T>1$, where $\mathbbm{1}_\mathcal{A}$ is the characteristic function of a set $\mathcal{A}$ and $\|x\|$ denotes the distance from $x$ to the nearest integer.
\end{theorem} 

\begin{theorem}\label{MAIN2}
Let $\mathcal{L}\in \mathcal{S}$ be such that $d_{\mathcal{L}}\geq2$, $\alpha\neq1$ be a complex number and $\epsilon\in(0,1/2)$.  Then
\begin{align*}
\mathop{\sum\limits_{0<\gamma_\alpha\leq T}}_{\beta_\alpha\geq0}x^{-\rho_\alpha}
=&\,-\dfrac{T}{2\pi x}\overline{\Lambda\left(x;\mathcal{L},0\right)}+O_\epsilon\left(x^\epsilon\left(1+\mathbbm{1}_{\mathbb{R}\setminus\mathbb{Z}}(x)\min\left\{\dfrac{T}{x},\dfrac{1}{\| x\|}\right\}\right)\right)+\\
&+O\left(x^{\epsilon}+\dfrac{1}{x^{\sigma\left(\mathcal{L},\alpha\right)+1}\log x}+x^{\epsilon}\log T\left(\dfrac{1}{\log x}+\log\dfrac{1}{\epsilon}\right)\right)
\end{align*}
for any $x,T>1$.
\end{theorem}

\begin{theorem}\label{MAIN3}
Let  $\mathcal{L}\in S$ be such that $d_\mathcal{L}=1$ and $\alpha\neq1$ be a complex number. 
Let also $m=1$ if $\mathcal{L}\equiv\zeta$ and $m=7$, otherwise.
Then
\begin{align*}
\mathop{\sum\limits_{0<\gamma_\alpha\leq T}}_{\beta_\alpha\geq0}x^{-\rho_\alpha}
=&-\dfrac{T}{2\pi x}\overline{\Lambda\left(x;\mathcal{L},0\right)}+O\left(\log x\min\left\{\dfrac{T}{x},\dfrac{1}{\langle x\rangle}\right\}\right)+O\left(\left(\log T\right)^{2m}\log(3x)\right)+\\
&+O\left((1+\log x)\log(2x)\log\log(3x)+\dfrac{1}{x^{\sigma\left(\mathcal{L},\alpha\right)+1}\log x}\right)+\\
&+O\left(\log T\left(\dfrac{1}{\log x}+\log x+\log\log(3x)\right)\right)
\end{align*}
for any $x,T>1$.
\end{theorem}

We do not treat the case of $\alpha=1$ since it seems to be rather complicated.
The major difficulty arises from the fact that the function on the right-hand side of \eqref{coeff} can not be represented as an ordinary Dirichlet series if $\alpha=1$ (this will become apparent in Lemma 1 of Section 3).
There have been attempts in \cite{JMS}, \cite{Reh},  and \cite{st} but in each case the aforementioned lack of an ordinary Dirichlet series represantation has been overseen. 
It was first Baluyot and Gonek \cite{bg} who noticed these inaccuracies and gave a partial answer also for the case of $\alpha=1$.

Explicit formulae in the case of $L$-functions in the Selberg class have been proven also by Murty and Perelli \cite{MP} with respect to their non-trivial zeros, as well as by Jakhlouti, Mazhouda and Steuding \cite{JMS} with respect to $\alpha$-points.
In particular those formulae in \cite{JMS} hold under some hypotheses which are not yet proven to be true.
Moreover, their formulae are not uniform in $x$ and there is a $T^{1/2+\epsilon}$ in the error term.
Our results are unconditional, uniform in $x$ and have smaller error terms.

As we mentioned earlier, such formulae are useful in the study of the vertical distribution of $\alpha$-points of a given function.
However, our poor knowledge of the horizontal distribution of $\alpha$-points of functions $\mathcal{L}$ in the Selberg class which have degree $d_\mathcal{L}\geq2$, prevents us from deriving results similar to the ones that have been proved for $\zeta(s)$.
We would have to assume certain hypotheses such as the Lindel\"of hypothesis for $L$-functions in the Selberg class.
For more details we refer to \cite{JMS} and its bibliography.
We are optimistic that our explicit formulae and these additional hypotheses can produce discrepancy estimates for sequences $(a\gamma_\alpha)_{\gamma_\alpha>0}$, but we will not persue this direction further.
Instead we prove an unconditional result regarding the vertical-distribution of such $\alpha$-points.

\begin{theorem}\label{Main1}
Let $\mathcal{L}\in\mathcal{S}\setminus\lbrace1\rbrace$ and $\alpha$ be a complex number.
Then there exists a positive constant $A=A(\alpha)$ such that for every $T\geq \exp(3)$ one can find an $\alpha$-point $\beta_\alpha+i\gamma_\alpha$ of $\mathcal{L}(s)$ satisfying
\begin{align*}
|\gamma_\alpha-T|<\dfrac{A}{\log\log\log T}.
\end{align*}
\end{theorem}
This result is a generalization of a theorem due to Littlewood \cite{ little } regarding the non-trivial zeros of $\zeta(s)$, and it yields the following corollary.

\begin{corollary}\label{Main2}
Let $\mathcal{L}\in\mathcal{S}\setminus\lbrace1\rbrace$, $b>0$ be a real number and $\alpha$ a complex number.
Then there exists a subsequence of $\alpha$-points $\left(\rho_{\alpha,n_k}\right)_{k\in\mathbb{N}}$, of $\mathcal{L}(s)$, such that $\gamma_{\alpha,n_k}=bk+o(1)$, and the sequence $\left(a \gamma_{\alpha, n_{k^m}}\right)_{k\in\mathbb{N}}$, is uniformly distributed mod 1 for every real number $a\notin b^{-1}\mathbb{Q}$ and every positive integer $m$.
\end{corollary} 

Lastly, with the aid of the latter corollary, we prove a theorem which combines the vertical-distibution of $\alpha$-points with the universality of $\zeta(s)$.
\begin{theorem}\label{universality}
Let $\mathcal{L}\in\mathcal{S}\setminus\lbrace{1}\rbrace$ and $\alpha\in\mathbb{C}$. 
Then there exists a subsequence of $\alpha$-points $\left(\rho_{\alpha,n_k}\right)_{k\in\mathbb{N}}$, of $\mathcal{L}(s)$ such that for any compact set with connected complement $K\subseteq\mathcal{D}:=\left\{s\in\mathbb{C}:1/2<\sigma<1\right\}$, any non-vanishing function $f$ which is continuous on $K$ and analytic in its interior, and any $\varepsilon>0$
$$\liminf\limits_{N\to\infty}\dfrac{1}{N}\sharp\left\{1\leq k\leq N:\max\limits_{s\in K}|\zeta\left(s+i\gamma_{\alpha, n_k}\right)-f(s)|<\varepsilon)\right\}>0,$$
where $\sharp A$ denotes the cardinality of a set $A\subseteq\mathbb{N}$.
\end{theorem}

In zeta-function theory, `universality' stands for an approximation property as in the previous theorem; this name has been coined with respect to the fact that a wide class of functions can be uniformly approximated by certain shifts of a single function. The first universality theorem is due to Voronin \cite{voro}; in his result there was no condition on the shifts. The first discrete universality theorem is due to Reich \cite{reich}; in his result the shifts are taken from an arithmetic progression. 

The statement of Theorem \ref{universality} is not new in the case of $\alpha=0$; however, different from previous versions, it is unconditional. First, Garunk\v stis, Laurin\v cikas and Macaitien\.e \cite{glm} proved the above universality theorem under assumption of what they call a {\it weak Montgomery conjecture} on the spacing of the imaginary parts of the nontrivial zeros (as it would follow from the pair correlation conjecture). Secondly, Garunk\v stis and Laurin\v cikas \cite{gl} obtained the same universality result assuming the Riemann hypothesis. It appears that the statement holds unconditionally and even in the wider context of $\alpha$-points of $L$-functions from the Selberg class.

\section{The Selberg Class: Definition and First Properties}

For a survey on the Selberg class and the value-distribution of its functions we refer to Kaczorowski and Perelli \cite{KP}, Murty and Murty \cite{MM}, and Perelli \cite{Per}.
Here we give a few definitions and poperties of $L$-functions from the Selberg class which is usually denoted by $\mathcal{S}$ and was introduced by Selberg \cite{sel}. 
It consists of Dirichlet series
\begin{align*}
\mathcal{L}(s):=\sum\limits_{n=1}^{\infty}\dfrac{f(n)}{n^s},
\end{align*}
satisfying the following hypotheses:
\begin{enumerate}
\item{\it Ramanujan hypothesis.} $f(n)\ll_\epsilon n^\epsilon$.
\item{\it Analytic continuation.} There exists a non-negative integer $k$ such that $(s-1)^k\mathcal{L}(s)$ is an entire function of finite order.
\item{\it Functional equation.} $\mathcal{L}(s)$ satisfies a functional equation of type 
\begin{align*}
\mathfrak{L}(s)=\omega\overline{\mathfrak{L}(1-\overline{s})},
\end{align*}
where
\begin{align*}
\mathfrak{L}(s):=\mathcal{L}(s)Q^s\prod\limits_{j=1}^{J}\Gamma\left(\lambda_js+\mu_j\right)
\end{align*}
with positive real numbers $Q$, $\lambda_j$, and complex numbers $\mu_j$, $\omega$ with $\Re\mu_j\geq0$ and $|\omega|=1$.
\item{\it Euler product.} $\mathcal{L}(s)$ has a product represantation
\begin{align*}
\mathcal{L}(s)=\prod\limits_{p}\mathcal{L}_p(s),
\end{align*} 
where 
\begin{align*}
\mathcal{L}_p(s)=\exp\left(\sum\limits_{k=1}^{\infty}\dfrac{b(p^k)}{p^{ks}}\right)
\end{align*}
with suitable coefficients $b(p^k)$ satisfying $b(p^k)\ll p^{k\theta}$ for some $\theta<1/2$.
\end{enumerate}
Axioms (i) and (ii) imply that a function $\mathcal{L}(s)$ from $\mathcal{S}$ is a Dirichlet series which is absolutely convergent for $\sigma>1$ and it has an analytic continuation to the whole complex plane except for a possible pole at $s=1$.
From axiom (iii) one obtains the quantity
$$d_{\mathcal{L}}:=2\sum\limits_{j=1}^J\lambda_j,$$
which is called the degree of $\mathcal{L}(s)$ and, although the data from the functional equation is not unique, $d_{\mathcal{L}}$ is well-defined (by its appearance in a zero-counting formula). 
It has been conjectured that the degree is always a non-negative integer. 
This deep conjecture has been verified so far only for $d_\mathcal{L}\leq 2$ by work of Kaczorowski and Perelli \cite{KP2}. 
The only element of degree zero is the function constant $1$;
 the degree one elements are the Riemann zeta-function $\zeta(s)$ in addition with real shifts of Dirichlet $L$-functions to primitive residue class characters. 
Typical examples of degree two are Dedekind zeta-functions to quadratic number fields and $L$-functions to normalized newforms. 
All these examples posses an Euler product and, indeed, axiom (iv) implies that if $\mathcal{L}\in\mathcal{S}$, then it is zero-free in the half-plane $\sigma>1$ and $f(n)$ is a multiplicative function.

The functional equation of a function from the Selberg class will play a crucial role in the proofs presented here.
We prefer to rewrite it as 
\begin{align}\label{functequa}
\mathcal{L}(s)=H_\mathcal{L}(s)\overline{\mathcal{L}(1-\overline{s})},
\end{align}
valid for all $s\in\mathbb{C}$, where 
\begin{align*}
H_\mathcal{L}(s)=\omega Q^{1-2s}\prod\limits_{j=1}^J\dfrac{\Gamma(\lambda_j(1-s)+\overline{\mu_j})}{\Gamma(\lambda_js+\mu_j)}.
\end{align*}
By Stirling's formula, it follows that
\begin{align}\label{Stril}
H_\mathcal{L}(\sigma+it)=\left(\lambda Q^2 t^{d_\mathcal{L}}\right)^{1/2-\sigma}\exp\left(-it\log\left(\lambda Q^2\left(\dfrac{t}{e}\right)^{d_\mathcal{L}}\right)\right)e^{i\pi(\mu-d_\mathcal{L})/4}\left(\omega+O_{\sigma_1,\sigma_2}\left(\dfrac{1}{t}\right)\right)
\end{align}
and
\begin{align}\label{fS}
-\dfrac{H'_\mathcal{L}}{H_\mathcal{L}}(\sigma+it)=\log\left(\lambda Q^2\left(\dfrac{t}{e}\right)^{d_\mathcal{L}}\right)+O_{\sigma_1,\sigma_2}\left(\dfrac{1}{t}\right),
\end{align}
uniformly in vertical strips $\sigma_1\leq\sigma\leq\sigma_2$, where 
$$\mu:=2\sum\limits_{j=1}^J(1-2\mu_j)\,\,\,\text{ and }\,\,\,\lambda:=\prod\limits_{j=1}^J\lambda_j^{2\lambda_j}.$$
In addition to axiom (iii), we have 
\begin{align}\label{order}
|\mathcal{L}(\sigma+it)|\asymp_{\sigma_1,\sigma_2}|t|^{(1/2-\sigma)d_\mathcal{L}}|\mathcal{L}(1-\sigma+it)|,\,\,\,|t|\geq t_0>0,
\end{align}
uniformly in $\sigma_1\leq\sigma\leq\sigma_2$. 

Concerning the distribution of the $\alpha$-points of a non-constant element $\mathcal{L}\in\mathcal{S}\setminus\lbrace1\rbrace$ of the Selberg class, it is worth to mention that their location in general is pretty similar to the case of the Riemann zeta-function. 
There may be an $\alpha$-point in the neighbourhood of a pole of a Gamma-factor in the functional equation (by Rouch\'e's theorem), however, those so-called trivial $\alpha$-points lie not too distant from the negative real axis and with finitely many further exceptions there are no other in the left half-plane $\sigma\leq0$. 
All other $\alpha$-points are said to be non-trivial and they lie in the right half-plane $\sigma\geq0$.
Their number $\mathcal{N}(\alpha,T)$ is asymptotically given by a Riemann-von Mangoldt--type formula: if $\alpha\neq1$ is a complex number, then
\begin{align}\label{RMf}
\mathcal{N}(\alpha,T):=\left\{s\in\mathbb{C}:\mathcal{L}(s)=\alpha\text{ and }T<t\leq2 T\right\}=\dfrac{d_\mathcal{L}}{2\pi}T\log\dfrac{4T}{e}+\dfrac{T}{2\pi}\log(\lambda Q^2)+O(\log T).
\end{align}
The proofs of \eqref{Stril}-\eqref{RMf} can be found in \cite[Lemma 6.7, Theorem 6.8, Corollary 7.4]{S}).

\section{Auxiliary Lemmas}

Before stating our first lemma we remind the Dirichlet convolution of two arithmetical functions.
If $f,g:\mathbb{N}\to\mathbb{C}$, then their Dirichlet convolution (or multiplication) is defined by
$$(f*g)(n):=\sum\limits_{d\mid n}f(d)g\left(\dfrac{n}{d}\right)$$
for all $n\in\mathbb{N}$. 
This operation is associative and we denote by $f^{* k}$ the Dirichlet convolution of $f$ with itself $k$ times for some $k\in\mathbb{N}$.
Lastly, we denote by $f_0$ the arithmetical function which is identical to $f$ except at $n=1$, where we define $f_0(1):=0$.

\begin{lemma}\label{Init}
Let $\mathcal{L}\in \mathcal{S}$ and $\alpha\neq1$ be a complex number.
 If 
 \begin{align}\label{aabscissa}
 \sigma(\mathcal{L},\alpha)
 :=
 \left\{ \begin{array}{ll}
1+\sup_{\rho_\alpha}\beta_\alpha,&\alpha\neq0,\\
 1,&\alpha=0.
 \end{array}
 \right.
 \end{align}
 then $1\leq\sigma(\mathcal{L},\alpha)<\infty$ and the function $\mathcal{L}'(s)/\left(\mathcal{L}(s)-\alpha\right)$ can be represented as an ordinary Dirichlet series which is absolutely convergent in the half-plane $\sigma>\sigma\left(\mathcal{L},\alpha\right)$. In particular,
\begin{align*}
\dfrac{\mathcal{L}'(s)}{\mathcal{L}(s)-\alpha}=-\sum\limits_{n=1}^{\infty}\dfrac{\Lambda(n;\mathcal{L},\alpha)}{n^s}
\end{align*}
for any $\sigma>\sigma\left(\mathcal{L},\alpha\right)$, where
\begin{align}\label{vMal}
\Lambda\left(n;\mathcal{L},\alpha\right):=\left(f\cdot\log\right)*\left(\sum\limits_{k=0}^{\infty}\dfrac{(-1)^kf_0^{*k}}{(1-\alpha)^{k+1}}\right)(n),
\end{align}
for any positive integer $n$.
Moreover, $\Lambda\left(n;\mathcal{L},\alpha\right)\ll_\epsilon n^{\sigma(\mathcal{L},\alpha)-1+\epsilon}$.
\end{lemma}

\begin{proof}
When $\alpha=0$ there is nothing to prove because for $\mathcal{L}\in S$ we know that
$$\dfrac{\mathcal{L}'(s)}{\mathcal{L}(s)}=-\sum\limits_{n=1}^{\infty}\dfrac{\Lambda(n;\mathcal{L},0)}{n^s},\,\,\,\sigma>1,$$
where $\Lambda(n;\mathcal{L},0)=f(n)\Lambda(n)$ and $\Lambda(n)$ is the von Mangoldt function.
Since $\Lambda(n)\ll\log n$, the Ramanujan hypothesis for the coefficients $f(n)$ imply that $\Lambda(n;\mathcal{L},0)\ll_\epsilon n^\epsilon$.

Now let $\alpha\neq0$.
Then
\begin{align*}
\mathcal{L}(s)-\alpha=(1-\alpha)\left(1+\sum\limits_{n=2}^\infty\dfrac{f(n)}{(1-\alpha)n^s}\right)
\end{align*}
is an ordinary Dirichlet series and the uniqueness theorem for Dirichlet series (see \cite[\S 9.6]{Titch1}) implies that there is an $A>0$ such that $\mathcal{L}(s)-\alpha\neq0$ for any $\sigma>A$.
Therefore, $$0\leq\sigma(\mathcal{L},\alpha)-1\leq A.$$
Moreover, it follows from \cite[Appendix, pg. 89-90]{Lan} that there is $B>0$ such that
\begin{align}\label{inv}
\dfrac{1}{\mathcal{L}(s)-\alpha}=\dfrac{1}{1-\alpha}\sum\limits_{k=0}^\infty(-1)^k\left(\sum\limits_{n=2}^\infty\dfrac{f(n)}{1-\alpha}\dfrac{1}{n^s}\right)^k=\sum\limits_{n=1}^\infty\left(\sum\limits_{k=0}^{\infty}\dfrac{(-1)^kf_0^{*k}(n)}{(1-\alpha)^{k+1}}\right)\dfrac{1}{n^s}
\end{align}
for any $\sigma>B$, while from \cite[Satz 12]{Lan} we know that the Dirichlet series on the right hand side of the latter relation converges for every $\sigma>\sigma(\mathcal{L},\alpha)-1$.
Hence, the aforementioned series is a Dirichlet series representation of $1/(\mathcal{L}(s)-\alpha)$ in the  half-plane $\sigma>\sigma(\mathcal{L},\alpha)-1$. 
This implies that it is absolutely convergent in the half-plane $\sigma>\sigma(\mathcal{L},\alpha)$ and that
\begin{align}\label{g}
g(n):=\sum\limits_{k=0}^{\infty}\dfrac{(-1)^kf_0^{*k}(n)}{(1-\alpha)^{k+1}}\ll_\epsilon n^{\sigma\left(\mathcal{L},\alpha\right)-1+\epsilon}.
\end{align}
On the other hand, the series
\begin{align}\label{der}
\mathcal{L}'(s)=-\sum\limits_{n=1}^{\infty}\dfrac{f(n)\log n}{n^s}
\end{align}
is absolutely convergent for any $\sigma>1$. 
The first assertion of the lemma follows now by mutliplying the Dirichlet series \eqref{inv} and \eqref{der}. 
In addition, relation \eqref{g} and the Ramanujan hypothesis yield that the coefficients of the resulting series satisfy
\begin{align*}
\Lambda\left(n;\mathcal{L},\alpha\right)=-\sum\limits_{d\mid n}g\left(\dfrac{n}{d}\right)f(d)\log d \ll_\epsilon \sum\limits_{d\mid n}\left(\dfrac{n}{d}\right)^{\sigma\left(\mathcal{L},\alpha\right)-1+\epsilon/2}d^{\epsilon/2}\ll_\epsilon n^{\sigma\left(\mathcal{L},\alpha\right)-1+\epsilon}.
\end{align*}
\end{proof}

\begin{lemma}\label{hori}
Let $\mathcal{L}\in\mathcal{S}$ and $\alpha$ be a complex number. 
Then
\begin{align*}
\dfrac{\mathcal{L}'(s)}{\mathcal{L}(s)-\alpha}=\sum\limits_{|\gamma_{\alpha}-t|\leq1}\dfrac{1}{s-\rho_\alpha}+O(\log|t|),\,\,\,|t|\geq1,
\end{align*}
uniformly in $-1\leq\sigma\leq\sigma(\mathcal{L},\alpha)$.
\end{lemma}
\begin{proof}
The proof follows in exactly the same way, except of some minor changes, as the one of \cite[Lemma 8]{GS} in the case of $\zeta(s)$.
\end{proof}

\begin{lemma}\label{lb}
Let $\mathcal{L}\in\mathcal{S}$ and $\sigma_1\leq\sigma_2<0$. 
Then
\begin{align*}
|\mathcal{L}(\sigma+it)|\gg_{\sigma_1,\sigma_2} |t|^{(1/2-\sigma)d_\mathcal{L}},\,\,\,|t|\geq t_0>0,
\end{align*}
uniformly in $\sigma_1\leq\sigma\leq\sigma_2$.
\end{lemma}
\begin{proof}
We refer to \cite[Chapter III, \S 2]{bes} for an introduction to the theory of analytic almost periodic functions and their elementary properties.
Since the Dirichlet series $\mathcal{L}(s)$ is uniformly convergent in the strip $1<1-\sigma_2\leq\sigma\leq1-\sigma_1$,
it is a uniformly almost periodic function in that strip.
In addition, $\mathcal{L}(s)$ is non-vanishing in $\sigma>1$ by assumption.
Hence,
$\inf\left\{\left|\mathcal{L}(s)\right|:1-\sigma_2\leq\sigma\leq1-\sigma_1\right\}>0$ by \cite[\S 2, Corollary 1]{bes}.
 Taking also into account relation \eqref{order} the lemma follows.
\end{proof}

%

The following lemma originates from the work of Levinson \cite[Lemma 3.3, Lemma 3.4]{Lev} and it has been improved by Gonek in \cite[Lemma 1, Lemma 2]{gon}.
Although the proof there is given for a fixed real number $a$, one can easily see that it holds uniformly in any finite interval $[a_1,a_2]$.
\begin{lemma}\label{Gonek}
Let $a_1, a_2$ be fixed real numbers such that $a_1\leq a_2$.
Then, for all  sufficiently large numbers $A,B\gg_{a_1,a_2}1$ with $A<B\leq2A$, we have that
\begin{align}\label{GoN}
\mathcal{I}(r,a,A,B):=\int\limits_{A}^B\exp\left(it\log\dfrac{t}{re}\right)\left(\dfrac{t}{2\pi}\right)^{a-1/2}\mathrm{d}t=M(r,A,B)+E(r,A,B),
\end{align}
uniformly in $a_1\leq a\leq a_2$,
where
\begin{align*}
M(r,a,A,B)=
\left\{
\begin{array}{ll}
(2\pi)^{1-a}r^ae^{-ir+i\pi/4},&A<r\leq B,\\
0,&\text{otherwise},
\end{array}
\right.
\end{align*}
and
\begin{align*}
E(r,a,A,B)\ll_{a_1, a_2} A^{a-1/2}+\dfrac{A^{a+1/2}}{|A-r|+A^{1/2}}+\dfrac{B^{a+1/2}}{|B-r|+B^{1/2}}.
\end{align*}
\end{lemma}

\begin{lemma}[Gallagher's lemma]\label{gall}
Let $T_0$ and $T\geq\delta>0$ be real numbers and $A$ be a finite subset of $\left[T_0+\delta/2, T+T_0-\delta/2\right]$. 
Define $N_\delta(x)=\sum_{t\in A,|t-x|<\delta}1$ and assume that $f(x)$ is a complex-valued continuous function on $\left[T_0,T+T_0\right]$ continuously differentiable on $\left(T_0,T+T_0\right)$. Then
\begin{align*}
\sum\limits_{t\in A}N_{\delta}^{-1}(t)|f(t)|^2\leq\dfrac{1}{\delta}\int\limits_{T_0}^{T+T_0}|f(x)|^2\mathrm{d}x+\left(\int\limits_{T_0}^{T+T_0}|f(x)|^2\mathrm{d}x\int\limits_{T_0}^{T+T_0}|f'(x)|^2\mathrm{d}x\right)^{1/2}.
\end{align*}
\end{lemma}

\begin{proof}
For a proof see \cite[Lemma 1.4]{M}.
\end{proof}

\section{Explicit Formulae}

\begin{proof}[Proof of Theorem \ref{MAIN1}]
Let $T_0:=\min\left\{1,\gamma_\alpha:\gamma_\alpha>0\right\}/2$. 
Such number exists as follows from the Riemann-von Mangoldt formula \eqref{RMf}. 
Recall the definition of $\sigma\left(\mathcal{L},\alpha\right)$ in \eqref{aabscissa} and
assume that $T$ is not an ordinate of an $\alpha$-point of $\mathcal{L}(s)$.
Since there are only finitely many $\alpha$-points in the plane $\left\{s\in\mathbb{C}:\sigma\leq0,\,t\geq T_0\right\}$, by the calculus of the residues, we have that
\begin{align}\label{init}
\begin{split}
\mathop{\sum\limits_{0<\gamma_\alpha\leq T}}_{\beta_\alpha\geq0}x^{\rho_\alpha}
&=\dfrac{1}{2\pi i}\left\{\,\int\limits_{\sigma\left(\mathcal{L},\alpha\right)+\epsilon+iT_0}^{\sigma\left(\mathcal{L},\alpha\right)+\epsilon+iT}+\int\limits_{\sigma\left(\mathcal{L},\alpha\right)+\epsilon+iT}^{-2+iT}+\int\limits_{-2+iT}^{-2+iT_0}+\int\limits_{-2+iT_0}^{\sigma\left(\mathcal{L},\alpha\right)+\epsilon+iT_0}\right\}\hspace{-1pt}x^s\dfrac{\mathcal{L}'(s)}{\mathcal{L}(s)-\alpha}\mathrm{d}s\hspace*{-1pt}+\hspace{-1pt}O(1)\\
&=:\mathcal{I}_1+\mathcal{I}_2+\mathcal{I}_3+\mathcal{I}_4+O(1).
\end{split}
\end{align}

We start by estimating $\mathcal{I}_1$. 
From Lemma 1 we know that the Dirichlet series expression of $\mathcal{L}'(s)/\left(\mathcal{L}(s)-\alpha\right)$ is absolutely convergent in the half-plane $\sigma\geq \sigma\left(\mathcal{L},\alpha\right)+\epsilon$ and $\Lambda\left(n;\mathcal{L},\alpha\right)\ll_\epsilon n^{\sigma(\mathcal{L},\alpha)-1+\epsilon/2}$.
 Therefore, interchanging integration and summation yields
 \begin{align}\label{I1}
 \begin{split}
 \mathcal{I}_1=-\dfrac{1}{2\pi}\sum\limits_{n=1}^{\infty}\Lambda(n;\mathcal{L},\alpha)\left(\dfrac{x}{n}\right)^{\sigma\left(\mathcal{L},\alpha\right)+\epsilon}\int\limits_{T_0}^T\left(\dfrac{x}{n}\right)^{it}\mathrm{d}t=-\dfrac{T-T_0}{2\pi}\Lambda(x;\mathcal{L},\alpha)+E,
\end{split}
 \end{align}
 where
 \begin{align*}
 E\hspace*{-1.1pt}\ll\hspace*{-1pt}\mathop{\sum\limits_{n=1}^{\infty}}_{n\neq x}\hspace*{-1pt}|\Lambda(n;\mathcal{L},\alpha)|\hspace*{-1pt}\left(\dfrac{x}{n}\right)^{\sigma\left(\mathcal{L},\alpha\right)+\epsilon}\min\hspace*{-1pt}\left\{\hspace*{-1pt}T,\left|\log\frac{x}{n}\right|^{-1}\right\}
 \hspace*{-1pt}\ll_\epsilon \mathop{\sum\limits_{n=1}^{\infty}}_{n\neq x}\dfrac{x^{\sigma(\mathcal{L},\alpha)+\epsilon}}{n^{1+\epsilon/2}}\min\hspace*{-1pt}\left\{\hspace*{-1pt}T,\dfrac{\max\left\{x,n\right\}}{|x-n|}\hspace*{-1pt}\right\}.
 \end{align*}
 Assume that $x$ is not an integer. 
 Let also $\lfloor x\rfloor$ and $\lceil x\rceil$, denote the largest integer which is less than or equal to $x$ and the smallest integer which is greater than or equal to $x$, respectively. 
 Splitting the latter sum into sums where the indices of their summands range over $n\leq x/2$, $n\geq 2x$, $x/2<n<\lfloor x\rfloor$,  $\lceil x\rceil<n<2x$, $n=\lfloor x\rfloor$ and $n=\lceil x\rceil$, we have that
 \begin{align}\label{E}
 \begin{split}
 E&\ll_\epsilon x^{\sigma(\mathcal{L},\alpha)+\epsilon}+ x^{\sigma(\mathcal{L},\alpha)+\epsilon/2}\mathop{\sum\limits_{\frac{x}{2}< n<2x}}_{n\neq\lfloor x\rfloor,\lceil x\rceil}\dfrac{1}{|x-n|}+x^{\sigma(\mathcal{L},\alpha)-1+\epsilon/2}\min\left\{T,\dfrac{x}{ x-\lfloor x\rfloor},\dfrac{\lceil x\rceil}{\lceil x\rceil-x}\right\}\\
 &\ll_\epsilon  x^{\sigma(\mathcal{L},\alpha)+\epsilon}+x^{\sigma(\mathcal{L},\alpha)+\epsilon/2}\log x+x^{\sigma(\mathcal{L},\alpha)-1+\epsilon/2}\min\left\{T,\dfrac{x}{\| x\|}\right\}\\
 &\ll_\epsilon x^{\sigma(\mathcal{L},\alpha)+\epsilon}+x^{\sigma(\mathcal{L},\alpha)+\epsilon/2}\min\left\{\frac{T}{x},\dfrac{1}{\|x\|}\right\}.
 \end{split}
 \end{align}
 If on the other hand $x$ is a positive integer, then we do not need to consider the cases $n=\lfloor x\rfloor$ and $n=\lceil x\rceil$ seperately, in which event we just have
 \begin{align}\label{error1}
 E\ll_\epsilon x^{\sigma(\mathcal{L},\alpha)+\epsilon}.
 \end{align}
 Since $T_0\Lambda(x;\mathcal{L},\alpha)\ll_\epsilon x^{\sigma(\mathcal{L},\alpha)-1+\epsilon/2}$, it follows from relations \eqref{I1}-\eqref{error1} that
 \begin{align}\label{I1f}
 \mathcal{I}_1=-\dfrac{T}{2\pi}\Lambda(x;\mathcal{L},\alpha)+O_\epsilon\left(x^{\sigma(\mathcal{L},\alpha)+\epsilon}\left(1+\mathbbm{1}_{\mathbb{R}\setminus\mathbb{Z}}(x)\min\left\{\frac{T}{x},\dfrac{1}{\| x\|}\right\}\right)\right).
 \end{align}
 
 We continue with estimating $\mathcal{I}_2$. 
 It follows from Lemma \ref{hori} that
 \begin{align}\label{I2}
2\pi i \mathcal{I}_2-\sum\limits_{|\gamma_\alpha-T|\leq 1}\,\int\limits_{\sigma\left(\mathcal{L},\alpha\right)+\epsilon+iT}^{-2+iT}\dfrac{x^{s}}{s-\rho_\alpha}\mathrm{d}s
\ll\log T\int\limits_{-2}^{\sigma\left(\mathcal{L},\alpha\right)+\epsilon}x^{\sigma}\mathrm{d}\sigma
\ll x^{\sigma\left(\mathcal{L},\alpha\right)+\epsilon}\dfrac{\log T}{\log x}.
 \end{align}
 Moreover, by the calculus of the residues, we know that
 \begin{align*}
 \int\limits_{\sigma\left(\mathcal{L},\alpha\right)+\epsilon+iT}^{-2+iT}\dfrac{x^{s}}{s-\rho_\alpha}\mathrm{d}s=&\,-2\pi i R(\rho_\alpha)+\\
 &+\left\{
 \int\limits_{\sigma\left(\mathcal{L},\alpha\right)+\epsilon+iT}^{\sigma\left(\mathcal{L},\alpha\right)+\epsilon+i(T+2)}+
 \int\limits_{\sigma\left(\mathcal{L},\alpha\right)+\epsilon+i(T+2)}^{-2+i(T+2)}+\int\limits_{-2+i(T+2)}^{-2+iT}\right\}\dfrac{x^{s}}{s-\rho_\alpha}\mathrm{d}s,
 \end{align*}
 where $R(\rho_\alpha)$ is 0 or 1, depending on whether $\rho_\alpha$ lies or not, respectively, in the rectangle with vertices $\sigma\left(\mathcal{L},\alpha\right)+\epsilon+iT$,  $\sigma\left(\mathcal{L},\alpha\right)+\epsilon+i(T+2)$, $-2+i(T+2)$ and $-2+iT$.
 Then
 \begin{align}\label{smal}
 \begin{split}
  \int\limits_{\sigma\left(\mathcal{L},\alpha\right)+\epsilon+iT}^{-2+iT}\dfrac{x^{s}}{s-\rho_\alpha}\mathrm{d}s
 & \ll 1+\int\limits_{T}^{T+2}\dfrac{x^{\sigma\left(\mathcal{L},\alpha\right)+\epsilon}}{\sqrt{\epsilon^2+(t-\gamma_\alpha)^2}}\mathrm{dt}+\dfrac{x^{\sigma\left(\mathcal{L},\alpha\right)+\epsilon}}{\log x}+\int\limits_{T}^{T+2}\dfrac{\mathrm{dt}}{x^2\sqrt{4+(t-\gamma_\alpha)^2}}\\
  &\ll_\epsilon x^{\sigma\left(\mathcal{L},\alpha\right)+\epsilon}.
  \end{split}
 \end{align}
 Since $\sum_{|\gamma_\alpha-T|\leq 1}1\ll\log T$ by relation \eqref{RMf}, it follows from \eqref{I2} and \eqref{smal} that
 \begin{align}\label{I2f}
 \mathcal{I}_2\ll_\epsilon  x^{\sigma\left(\mathcal{L},\alpha\right)+\epsilon}\log T\left(1+\dfrac{1}{\log x}\right).
 \end{align}
 
In order to estimate $\mathcal{I}_3$, we firstly observe that Lemma \ref{lb} implies that
$$|\mathcal{L}(-2+it)|>2 |\alpha|$$
for any $t\geq T_1$, where $T_1\gg 1$ is a sufficiently large positive number such that it is not an ordinate of an $\alpha$-point of $\mathcal{L}(s)$.
Therefore, for $T\geq T_1$ we have that
\begin{align}\label{I3}
\begin{split}
2\pi x^2\mathcal{I}_3+O(1)
&=-\int\limits_{T_1}^Tx^{it}\dfrac{\mathcal{L}'}{\mathcal{L}}(-2+it)\left(1-\frac{\alpha}{\mathcal{L}(-2+it)}\right)^{-1}\mathrm{d}t\\
&=-\int\limits_{T_1}^Tx^{it}\dfrac{\mathcal{L}'}{\mathcal{L}}(-2+it)\left(1+\sum\limits_{k=1}^{\infty}\left(\dfrac{\alpha}{\mathcal{L}(-2+it)}\right)^k\right)\mathrm{d}t\\
&=:-\left(\mathcal{I}_{31}+\mathcal{I}_{32}\right).
\end{split}
\end{align}
In view of the functional equation \eqref{functequa} and relation \eqref{fS}, it follows that
\begin{align}\label{forla}
\begin{split}
\mathcal{I}_{31}
&=\int\limits_{T_1}^Tx^{it}\left(-\dfrac{\overline{\mathcal{L}}'}{\overline{\mathcal{L}}}(3+it)+\dfrac{H_\mathcal{L}'}{H_\mathcal{L}}(-2+it)\right)\mathrm{d}t\\
&=\sum\limits_{n=1}^{\infty}\dfrac{\overline{\Lambda\left(n;\mathcal{L},0\right)}}{n^{3}}\int\limits_{T_1}^T(xn)^{it}\mathrm{d}t-\int\limits_{T_1}^Tx^{it}\log\left(\lambda Q^2\left(\dfrac{t}{e}\right)^{d_\mathcal{L}}\right)\mathrm{d}t+\int\limits_{T_1}^Tx^{it}O\left(\dfrac{1}{t}\right)\mathrm{d}t\\
&=:\mathcal{I}_{3111}+\mathcal{I}_{312}+\mathcal{I}_{313}.
\end{split}
\end{align}
Observe that
\begin{align}
\mathcal{I}_{3111}\ll\min\left\{T,\dfrac{1}{\log x}\right\}\,\,\,\text{ and }\,\,\,\mathcal{I}_{313} \ll\log T.
\end{align}
To estimate $\mathcal{I}_{312}$ we integrate by parts. 
Then
\begin{align}\label{forla1}
\mathcal{I}_{312}=\left[-\dfrac{x^{it}}{\log x}\log\left(\lambda Q^2\left(\dfrac{t}{e}\right)^{d_\mathcal{L}}\right)\right]_{T_1}^T+\int\limits_{T_1}^T\dfrac{x^{it}ed_\mathcal{L}}{t\log x}\mathrm{d}t\ll\dfrac{\log T}{\log x}.
\end{align}
Lastly, Lemma \ref{lb} implies that
\begin{align}\label{I32}
\mathcal{I}_{32}\ll \int\limits_{T_1}^{T}\sum\limits_{k=1}^{\infty}\left(\frac{\alpha}{t^{3d_\mathcal{L}/2}}\right)^k\mathrm{d}t\ll\int\limits_{T_1}^{T}\dfrac{1}{t^{5/2}}\mathrm{d}t\ll1.
\end{align}
In view of relations \eqref{I3}-\eqref{I32}, we obtain that
\begin{align}\label{I3f}
\mathcal{I}_3\ll\dfrac{1}{x^2}\left(\min\left\{T,\dfrac{1}{\log x}\right\}+\log T\left(1+\dfrac{1}{\log x}\right)\right),
\end{align}
as long as $T\geq T_1$.
For the remaining $T$, that is for $T_0\leq T<T_1$, we can bound $\mathcal{I}_3$ from above trivially by $x^{-2}$ which is absorbed  by the error term appearing in \eqref{I1f} since $x>1$.

It remains to estimate $\mathcal{I}_4$ which is independent of $T$ and, therefore, can be bounded trivially from above by $x^{\sigma\left(\mathcal{L},\alpha\right)+\epsilon}$.

The theorem now follows from \eqref{init}, \eqref{I1f}, \eqref{I2f} and \eqref{I3f}. 
Recall that we assumed in the beginning that $T$ is not an ordinate of an $\alpha$-point.
Otherwise, we choose a $T_2\in[T,T+1]$ such that it is not an ordinate of an $\alpha$-point and observe that
$$\mathop{\sum\limits_{0<\gamma_\alpha\leq T}}_{\beta_\alpha\geq0}x^{\rho_\alpha}=\mathop{\sum\limits_{0<\gamma_\alpha\leq T_2}}_{\beta_\alpha\geq0}x^{\rho_\alpha}+O\left(x^{\sigma\left(\mathcal{L},\alpha\right)}\sum\limits_{T\leq\gamma_\alpha\leq T+1}1\right)
=\mathop{\sum\limits_{0<\gamma_\alpha\leq T_2}}_{\beta_\alpha\geq0}x^{\rho_\alpha}+O\left(x^{\sigma\left(\mathcal{L},\alpha\right)}\log T\right).$$
Then we estimate the latter sum as described above, while the error term is absorbed by the error term appearing in \eqref{I2f}.
\end{proof}

\begin{proof}[Proof of Theorem \ref{MAIN2}]
Let $T_0$ be as in the proof of Theorem \ref{MAIN1} and $\epsilon\in(0,1/2)$.
If $T$ is not an ordinate of an $\alpha$-point,
then by the calculus of the residues it follows that
\begin{align}\label{init1}
\begin{split}
\mathop{\sum\limits_{0<\gamma_\alpha\leq T}}_{\beta_\alpha\geq0}x^{-\rho_\alpha}
=&\dfrac{1}{2\pi i}\left\{\,\int\limits_{\sigma\left(\mathcal{L},\alpha\right)+1+iT_0}^{\sigma\left(\mathcal{L},\alpha\right)+1+iT}+\int\limits_{\sigma\left(\mathcal{L},\alpha\right)+1+iT}^{-\epsilon+iT}+\int\limits_{-\epsilon+iT}^{-\epsilon+iT_0}+\int\limits_{-\epsilon+iT_0}^{\sigma\left(\mathcal{L},\alpha\right)+1+iT_0}\right\}\hspace{-1pt}x^{-s}\dfrac{\mathcal{L}'(s)}{\mathcal{L}(s)-\alpha}\mathrm{d}s+\\
&+O(x^\epsilon)\\
=:&\,\mathcal{I}_1+\mathcal{I}_2+\mathcal{I}_3+\mathcal{I}_4+O(x^\epsilon).
\end{split}
\end{align}
The error term occurs from the possible existence of finitely many $\alpha$-points in the vertical strip $-\epsilon\leq\sigma<0$.

The absolute convergence of $\mathcal{L}'(s)/(\mathcal{L}(s)-\alpha)$ in the half-plane $\sigma\geq\sigma(\mathcal{L},\alpha)+1$ and our assumption that $x>1$ imply that
\begin{align}\label{I1i}
2\pi x^{\sigma\left(\mathcal{L},\alpha\right)+1}\mathcal{I}_1=\sum\limits_{n=1}^{\infty}\dfrac{\Lambda\left(n;\mathcal{L},\alpha\right)}{n^{\sigma\left(\mathcal{L},\alpha\right)+1}}\int\limits_{T_0}^T(xn)^{-it}\mathrm{d}t\ll\sum\limits_{n=1}^{\infty}\dfrac{|\Lambda\left(n;\mathcal{L},\alpha\right)|}{n^{\sigma\left(\mathcal{L},\alpha\right)+1}\log(xn)}\ll\dfrac{1}{\log x}.
\end{align}

Once more, $\mathcal{I}_4$ is easily seen to be bounded by $x^\epsilon$ which is abosrbed in the error term appearing in \eqref{init1}.

To estimate $\mathcal{I}_2$ we work as in \eqref{I2} and \eqref{smal}, where instead of $x^s$ we have $x^{-s}$.
Then
\begin{align}\label{I2i}
\begin{split}
\mathcal{I}_2
&\ll\sum\limits_{|\gamma_\alpha-T|\leq 1}\left(1+\int\limits_{T}^{T+2}\dfrac{x^{-\left(\sigma\left(\mathcal{L},\alpha\right)+1\right)}}{\sqrt{1+(t-\gamma_\alpha)^2}}\mathrm{dt}+\dfrac{x^\epsilon}{\log x}+\int\limits_{T}^{T+2}\dfrac{x^\epsilon}{\sqrt{\epsilon^2+(t-\gamma_\alpha)^2}}\mathrm{d}t\right)+x^\epsilon\dfrac{\log T}{\log x}\\
&\ll\sum\limits_{|\gamma_\alpha-T|\leq 1}\left(1+\dfrac{x^\epsilon}{\log x}+x^\epsilon\int\limits_{T}^{T+2}\min\left\{\frac{1}{\epsilon},\dfrac{1}{|t-\gamma_\alpha|}\right\}\mathrm{d}t\right)+x^\epsilon\dfrac{\log T}{\log x}\\
&\ll x^\epsilon\log T\left(\dfrac{1}{\log x}+\log\dfrac{1}{\epsilon}\right).
\end{split}
\end{align} 
It remains to estimate $\mathcal{I}_3$.
Lemma \ref{lb} implies that
\begin{align*}
|\mathcal{L}(-\epsilon+it)|>2|\alpha|
\end{align*}
for any $t\geq T_1$, where $T_1\gg_\epsilon1$ is a sufficienlty large positive number such that it is not an ordinate of an $\alpha$-point.
Therefore, for $T\geq T_1$ we have that
\begin{align}\label{I3i}
\begin{split}
\dfrac{2\pi}{x^\epsilon}\mathcal{I}_3+O(1)
&=-\int\limits_{T_1}^Tx^{-it}\dfrac{\mathcal{L}'}{\mathcal{L}}(-\epsilon+it)\left(1-\frac{\alpha}{\mathcal{L}(-\epsilon+it)}\right)^{-1}\mathrm{d}t\\
&=-\int\limits_{T_1}^Tx^{-it}\dfrac{\mathcal{L}'}{\mathcal{L}}(-\epsilon+it)\left(1+\sum\limits_{k=1}^{\infty}\left(\dfrac{\alpha}{\mathcal{L}(-\epsilon+it)}\right)^k\right)\mathrm{d}t\\
&=:-\left(\mathcal{I}_{31}+\mathcal{I}_{32}\right).
\end{split}
\end{align}
Working as in \eqref{forla}-\eqref{forla1}, it follows that
\begin{align*}
\mathcal{I}_{31}=\sum\limits_{n=1}^{\infty}\dfrac{\overline{\Lambda\left(n;\mathcal{L},0\right)}}{n^{1+\epsilon}}\int\limits_{T_1}^T\left(\dfrac{n}{x}\right)^{it}\mathrm{d}t+O\left(\log T\left(1+\dfrac{1}{\log x}\right)\right).
\end{align*}
By estimating the latter sum as in \eqref{I1}-\eqref{error1}, we obtain that
\begin{align}\label{I31i}
\mathcal{I}_{31}=\dfrac{T}{x^{1+\epsilon}}\overline{\Lambda\left(x;\mathcal{L},0\right)}+O_\epsilon\left(1+\dfrac{\mathbbm{1}_{\mathbb{R}\setminus\mathbb{Z}}(x)}{x^{\epsilon/2}}\min\left\{\dfrac{T}{x},\dfrac{1}{\| x\|}\right\}\right)+O\left(\log T\left(1+\dfrac{1}{\log x}\right)\right)
\end{align}

To estimate $\mathcal{I}_{32}$ we simply note that $d_{\mathcal{L}}\geq2$.
Thus, in view of Lemma \ref{lb}, we have that
\begin{align}\label{I32i}
\mathcal{I}_{32}
\ll_\epsilon\int\limits_{T_1}^T\sum\limits_{k=1}^{\infty}\left(\dfrac{\alpha}{t^{(1/2+\epsilon)d_\mathcal{L}}}\right)^k\mathrm{d}t
\ll_\epsilon\int\limits_{T_1}^T\dfrac{1}{t^{1+2\epsilon}}\ll_\epsilon 1.
\end{align}
Combining relations \eqref{I3i}-\eqref{I32i} yields that
\begin{align}\label{I3if}
\mathcal{I}_3
=-\dfrac{T\overline{\Lambda\left(x;\mathcal{L},0\right)}}{2\pi x}+O_\epsilon\hspace*{-2pt}\left(x^\epsilon\left(1+\mathbbm{1}_{\mathbb{R}\setminus\mathbb{Z}}(x)\min\hspace*{-0.7pt}\left\{\dfrac{T}{x},\dfrac{1}{\| x\|}\right\}\right)\hspace*{-1pt}\right)+O\hspace*{-2pt}\left(x^{\epsilon}\log T\left(1+\dfrac{1}{\log x}\right)\hspace*{-1pt}\right),
\end{align}
as long as $T\geq T_1$.
For the remaining $T$, that is for $T_0\leq T<T_1$, we have that
\begin{align}\label{remain}
\mathcal{I}_3=\pm\dfrac{T\overline{\Lambda\left(x;\mathcal{L},0\right)}}{2\pi x}+\int\limits_{-\epsilon+iT}^{-\epsilon+iT_0}x^{-s}\dfrac{\mathcal{L}'(s)}{\mathcal{L}(s)-\alpha}\mathrm{d}s=-\dfrac{T\overline{\Lambda\left(x;\mathcal{L},0\right)}}{2\pi x}+O_\epsilon\left(x^\epsilon\right).
\end{align}

The theorem now follows from relations \eqref{init1}-\eqref{I2i}, \eqref{I3if} and \eqref{remain} and by arguing similarly as in the end of the proof of the Theorem \ref{MAIN1} for the case of $T$ being an ordinate of an $\alpha$-point.
\end{proof}

\begin{proof}[Proof of Theorem \ref{MAIN3}]
We set $\epsilon=1/\log 3x$ and we work exactly as in the proof of Theorem \ref{MAIN2} to estimate $\mathcal{I}_1$, $\mathcal{I}_2$ and $\mathcal{I}_4$.
In particular, we have that $1\leq x^{\epsilon}\leq e$ and
\begin{align}\label{firste}
\mathcal{I}_1\ll\dfrac{1}{x^{\sigma\left(\mathcal{L},\alpha\right)+1}\log x},\,\,\,\mathcal{I}_2\ll\log T\left(\dfrac{1}{\log x}+\log\log (3x)\right)\,\,\,\text{ and }\,\,\,\mathcal{I}_4\ll1.
\end{align}

The estimation of $\mathcal{I}_3$ is more laborious.
Since $d_\mathcal{L}=1$, $\mathcal{L}(s)$ is either the Riemann zeta-function $\zeta(s)$ or a shift $L\left(s+i\theta;\chi\right)$ of a Dirichlet $L$-function for some $\theta\in\mathbb{R}$ and some primitive Dirichlet character $\chi$.
Then, it follows from \cite[Theorem 3.22, Theorem 8.20, Theorem 8.22]{Ten} that
\begin{align}
|\mathcal{L}(\sigma+it)|\gg\left(\log |t|\right)^{-m},\,\,\,|t|\geq t_0>1,
\end{align}
uniformly in $1\leq\sigma\leq2$, where $m=1$ if $\mathcal{L}(s)=\zeta(s)$ and $m=7$, otherwise.
In view of relation \eqref{order} we deduce that
\begin{align}\label{gorder}
|\mathcal{L}(\sigma+it)|\gg\dfrac{|t|^{1/2-\sigma}}{\left(\log |t|\right)^{m}},\,\,\,|t|\geq t_0>1,
\end{align}
uniformly in $-1\leq\sigma\leq0$.
Hence, there is a sufficiently large positive number $T_1$, only this time independent of $\epsilon$, such that it is not an ordinate of an $\alpha$-point and
$$|\mathcal{L}(\sigma+it)|>2|\alpha|$$
for any $t\geq T_1$ and  $-1\leq\sigma\leq0$.
Therefore, for $T\geq T_1$ we have that
\begin{align}\label{I3ii}
\begin{split}
\frac{2\pi}{ x^{\epsilon}}\mathcal{I}_3+O(1)
&=-\int\limits_{T_1}^Tx^{-it}\dfrac{\mathcal{L}'}{\mathcal{L}}(-\epsilon+it)\left(1-\frac{\alpha}{\mathcal{L}(-\epsilon+it)}\right)^{-1}\mathrm{d}t\\
&=-\int\limits_{T_1}^Tx^{-it}\dfrac{\mathcal{L}'}{\mathcal{L}}(-\epsilon+it)\left(1+\dfrac{\alpha}{\mathcal{L}(-\epsilon+it)}+\sum\limits_{k=2}^{\infty}\left(\dfrac{\alpha}{\mathcal{L}(-\epsilon+it)}\right)^k\right)\mathrm{d}t\\
&=:-\left(\mathcal{I}_{31}+\mathcal{I}_{32}+\mathcal{I}_{33}\right).
\end{split}
\end{align}

Once more, we have that
\begin{align}\label{I31ii}
\mathcal{I}_{31}=\sum\limits_{n=1}^{\infty}\dfrac{\overline{\Lambda\left(n;\mathcal{L},0\right)}}{n^{1+\epsilon}}\int\limits_{T_1}^T\left(\dfrac{n}{x}\right)^{it}\mathrm{d}t+O\left(\log T\left(1+\dfrac{1}{\log x}\right)\right).
\end{align}
In this case, however, we know that $|\Lambda\left(n;\mathcal{L},0\right)|=\Lambda(n)$.
Thus,
\begin{align}\label{I31part}
\begin{split}
\sum\limits_{n=1}^{\infty}\dfrac{\overline{\Lambda\left(n;\mathcal{L},0\right)}}{n^{1+\epsilon}}\int\limits_{T_1}^T\left(\dfrac{n}{x}\right)^{it}\mathrm{d}t\
&=\frac{T-T_1}{x^{1+\epsilon}}\overline{\Lambda\left(x;\mathcal{L},0\right)}+\mathop{\sum\limits_{n=1}^{\infty}}_{n\neq x}\dfrac{\Lambda(n)}{n^{1+\epsilon}}\min\left\{T,\left|\log\dfrac{x}{n}\right|^{-1}\right\}\\
&=\hspace*{-0.55pt}\dfrac{T\overline{\Lambda\left(x;\mathcal{L},0\right)}}{x^{1+\epsilon}}
\hspace*{-0.55pt}+\hspace*{-0.55pt}O\left(\log(2 x)\log\log (3x)+\log x\min\hspace*{-0.5pt}\left\{\dfrac{T}{x},\dfrac{1}{\langle x\rangle}\right\}\right),
\end{split}
\end{align}
as follows by \cite[Lemma 2]{Gon}.

It is simple to estimate $\mathcal{I}_{33}$ since by relation \eqref{gorder} it follows that
\begin{align}\label{I33i}
\mathcal{I}_{33}\ll\int\limits_{T_1}^{T}\sum
\limits_{k=2}^{\infty}\left(\dfrac{\alpha(\log t)^m}{t^{1/2+\epsilon}}\right)^k\mathrm{d}t\ll(\log T)^{2m}\int\limits_{T_1}^T\dfrac{1}{t^{1+2\epsilon}}\mathrm{d}t\ll(\log T)^{2m}\log(3x).
\end{align}

To estimate $\mathcal{I}_{32}$ we start with integration by parts, that is,
\begin{align}\label{I32e}
\dfrac{x^{\epsilon}}{\alpha i}\mathcal{I}_{32}=-\int\limits_{-\epsilon+iT_1}^{-\epsilon+iT}x^{-s}\dfrac{\mathcal{L}'}{\mathcal{L}^2}(s)\mathrm{d}s
=\left[\dfrac{x^{-s}}{\mathcal{L}(s)}\right]_{-\epsilon+iT_1}^{-\epsilon+iT}+\log x\int\limits_{-\epsilon+iT_1}^{-\epsilon+iT}\dfrac{x^{-s}}{\mathcal{L}(s)}\mathrm{d}s.
\end{align}
The first term on the right-hand side is bounded above by $T_1^{-1/2}(\log T_1)^m\ll1$ by relation \eqref{gorder}.
For the second term we split the path of integration in a dyadic manner:
\begin{align}\label{split}
\int\limits_{-\epsilon+iT_1}^{-\epsilon+iT}\dfrac{x^{-s}}{\mathcal{L}(s)}\mathrm{d}s=O(1)+\sum\limits_{k=1}^{\left\lfloor\log (T/T_0)/\log 2\right\rfloor}\int\limits_{-\epsilon+iT/2^{k}}^{-\epsilon+iT/2^{k-1}}\dfrac{x^{-s}}{\mathcal{L}(s)}\mathrm{d}s=:O(1)+\sum\limits_{k=1}^{\left\lfloor\log (T/T_1)/\log 2\right\rfloor}I_k.
\end{align}
 Observe that $2^k\leq T$ for all $k=1,\dots,\left\lfloor\log (T/T_1)/\log 2\right\rfloor$.
 To estimate $I_{k}$ we use the functional equation  \eqref{functequa} of $\mathcal{L}(s)$:
 \begin{align*}
I_{k}
=\int\limits_{-\epsilon+iT/2^k}^{-\epsilon+iT/2^{k-1}}\dfrac{x^{-s}}{H_\mathcal{L}(s)\overline{\mathcal{L}(1-\overline{s})}}\mathrm{d}s
 =ix^{\epsilon} \int\limits_{T/2^k}^{T/2^{k-1}}\dfrac{x^{-it}}{H_{\mathcal{L}}(-\epsilon+it)\overline{\mathcal{L}(1+\epsilon+it)}}\mathrm{d}t.
 \end{align*}
 Since $\mathcal{L}(s)$ is absolutely convergent in the half-plane $\sigma>1$ and $f(n)$ is a completely multiplicative arithmetical function, we have that 
 \begin{align*}
 \dfrac{1}{\mathcal{L}(s)}=\sum\limits_{n=1}^{\infty}\dfrac{\mu(n)f(n)}{n^s},\,\,\,\sigma>1,
 \end{align*}
 where $\mu(n)$ is the M\"obius function and the series converges absolutely in the half-plane $\sigma>1$ (see for example \cite[Section 11.4, Example 3]{apo}.
 In addition to relation \eqref{Stril}, it follows that
 \begin{align*}
I_{k}
=&\,\dfrac{ix^{\epsilon}e^{i\pi(1-\mu)/4}}{\left(2\pi\lambda Q^2\right)^{1/2+\epsilon}}\sum\limits_{n=1}^{\infty}\dfrac{\mu(n)\overline{f(n)}}{n^{1+\epsilon}}\times\\
&\times\int\limits_{T/2^k}^{T/2^{k-1}}\exp\left(id_\mathcal{L}t\log\left(\left(\dfrac{n\lambda Q^2}{x}\right)^{1/d_\mathcal{L}}\dfrac{t}{e}\right)\right)\left(\dfrac{t}{2\pi}\right)^{-\epsilon-1/2}\dfrac{1}{\omega+O\left(\frac{1}{t}\right)}\mathrm{d}t.
\end{align*}
Observe that the implicit constant in the latter relation is independent of $\epsilon$ since $0<\epsilon<1/\log3$ and we applied relation \eqref{Stril} in the strip $-1/\log 3\leq\sigma\leq0$.
If we take now $T_1\gg1$ to be a sufficiently large positive number then
$$\left(\omega+O\left(\dfrac{1}{t}\right)\right)^{-1}=\dfrac{1}{\omega}+O\left(\dfrac{1}{t}\right)$$
for any $t\geq T_1$. 
Then integrating by substitution will yield
\begin{align}\label{Ik2}
 I_{k}
 \ll &\sum\limits_{n=1}^{\infty}\dfrac{1}{n^{1+\epsilon}}\left|\,\int\limits_{d_\mathcal{L}T/2^k}^{d_\mathcal{L}T/2^{k-1}}\exp\left(it\log\left(\left(\dfrac{n\lambda Q^2}{x}\right)^{1/d_\mathcal{L}}\dfrac{t}{d_\mathcal{L}e}\right)\right)\left(\dfrac{t}{2\pi}\right)^{-\epsilon-1/2}\mathrm{d}t\right|+\\
 &+\sum\limits_{n=1}^{\infty}\dfrac{1}{n^{1+\epsilon}}\int\limits_{T/2^k}^{T/2^{k-1}}\dfrac{1}{t^{\epsilon+3/2}}\mathrm{d}t,
 \end{align}
 since $\left|\mu(n){f(n)}\right|=1$ for all $n\in\mathbb{N}$.
 From the last term of the relation above we get 
 \begin{align}\label{Ik2p}
 \sum\limits_{n=1}^{\infty}\dfrac{1}{n^{1+\epsilon}}\int\limits_{T/2^k}^{T/2^{k-1}}\dfrac{1}{t^{\epsilon+3/2}}\mathrm{d}t
\ll \left(\dfrac{2^k}{T}\right)^{1/2}\log\log(3x).
 \end{align}
 For the first term of \eqref{Ik2}, which we denote by $J_k$, we employ Lemma \ref{Gonek}.
 To this end, we consider first $T_1\gg1$ sufficiently large, with the implicit constant being independent of $\epsilon$ as it ranges over the interval $(0,1/\log3)$.
 Then, following our notation from Lemma \ref{Gonek}, we obtain that
 \begin{align*}
 J_k=&\mathop{\sum\limits_{n=1}^{\infty}}_{\frac{d_\mathcal{L}T}{2^{k}}<d_\mathcal{L}\left(\frac{x}{\lambda Q^2n}\right)^{1/d_\mathcal{L}}\leq \frac{d_\mathcal{L}T}{2^{k-1}}}\dfrac{1}{n^{1+\epsilon}}M\left(d_\mathcal{L}\left(\dfrac{x}{\lambda Q^2n}\right)^{1/d_\mathcal{L}},-\epsilon,\dfrac{d_\mathcal{L}T}{2^{k}},\dfrac{d_\mathcal{L}T}{2^{k-1}}\right)+\\
 &+\sum\limits_{n=1}^{\infty}\dfrac{1}{n^{1+\epsilon}}E\left(d_\mathcal{L}\left(\dfrac{x}{\lambda Q^2n}\right)^{1/d_\mathcal{L}},-\epsilon,\dfrac{d_\mathcal{L}T}{2^{k}},\dfrac{d_\mathcal{L}T}{2^{k-1}}\right)\\
 &=:M_k+E_k.
 \end{align*}
 We estimate firstly $M_k$:
 \begin{align*}
M_k\ll\mathop{\sum\limits_{n=1}^{\infty}}_{\frac{d_\mathcal{L}T}{2^{k}}<d_\mathcal{L}\left(\frac{x}{\lambda Q^2n}\right)^{1/d_\mathcal{L}}\leq \frac{d_\mathcal{L}T}{2^{k-1}}}\dfrac{1}{n^{1+\epsilon}}\left(\dfrac{x}{\lambda Q^2n}\right)^{-\epsilon/d_\mathcal{L}}
\ll\sum\limits_{\left(\frac{2^{k-1}}{T}\right)^{d_\mathcal{L}}\frac{x}{\lambda Q^2}\leq n<\left(\frac{2^{k}}{T}\right)^{d_\mathcal{L}}\frac{x}{\lambda Q^2}}\dfrac{1}{n}
\ll 1.
 \end{align*}
We estimate now $E_k$:
 \begin{align*}
 E_k&\ll\sum\limits_{n=1}^{\infty}\dfrac{1}{n^{1+\epsilon}}\left[\left(\dfrac{d_\mathcal{L}T}{2^k}\right)^{-\epsilon-1/2}+\dfrac{\left(\frac{d_\mathcal{L}T}{2^k}\right)^{-\epsilon+1/2}}{\left|\frac{d_\mathcal{L}T}{2^k}-d_\mathcal{L}\left(\frac{x}{\lambda Q^2n}\right)^{1/d_\mathcal{L}}\right|+\left(\frac{d_\mathcal{L}T}{2^k}\right)^{1/2}}+\right.\\&\hspace*{7cm}\left.+\dfrac{\left(\frac{d_\mathcal{L}T}{2^{k-1}}\right)^{-\epsilon+1/2}}{\left|\frac{d_\mathcal{L}T}{2^{k-1}}-d_\mathcal{L}\left(\frac{x}{\lambda Q^2n}\right)^{1/d_\mathcal{L}}\right|+\left(\frac{d_\mathcal{L}T}{2^{k-1}}\right)^{1/2}}\right]\\
 &\ll\left(\dfrac{2^k}{T}\right)^\epsilon\sum\limits_{n=1}^{\infty}\dfrac{1}{n^{1+\epsilon}}\left[\left(\dfrac{2^k}{T}\right)^{1/2}+1+2^{-\epsilon}\right]\\
 &\ll\left(\dfrac{2^k}{T}\right)^{\epsilon}\log\log(3x).
 \end{align*}
 Summing up we will have that
 \begin{align*}
 \sum\limits_{k=1}^{\left\lfloor\log (T/T_1)/\log 2\right\rfloor}J_k
 &\ll\sum\limits_{k=1}^{\left\lfloor\log (T/T_1)/\log 2\right\rfloor}\left(1+\left(\dfrac{2^k}{T}\right)^{\epsilon}\log\log(3x)\right)\\
 &\ll\log T+\dfrac{2^{\epsilon\log (T/T_1)/\log 2}}{(2^\epsilon-1)T^\epsilon}\log\log(3x)\\
 &\ll\log T+\log(3x)\log\log(3x).
 \end{align*}
From this and relations \eqref{Ik2} and \eqref{Ik2p} it follows that
\begin{align*}
 \sum\limits_{k=1}^{\left\lfloor\log (T/T_1)/\log 2\right\rfloor}I_{k}
 \ll\,& \sum\limits_{k=1}^{\left\lfloor\log (T/T_1)/\log 2\right\rfloor}\left(\dfrac{2^{k}}{T}\right)^{1/2}\log\log(3x)+\log T+\log(3x)\log\log(3x)\\
 \ll\,&\log T+\log(2x)\log\log(3x).
\end{align*}
Then relations \eqref{I32e} and \eqref{split} imply that
\begin{align*}
\mathcal{I}_{32}\ll1+\log x\left(\log T+\log(2x)\log\log(3x)\right)
\end{align*}
and, thus, from \eqref{I31ii}-\eqref{I33i} we have that
\begin{align}\label{laste}
\begin{split}
\mathcal{I}_3=&\,-\dfrac{T\overline{\Lambda\left(x;\mathcal{L},0\right)}}{2\pi x}
+O\left(\log(2 x)\log\log (3x)+\log x\min\left\{\dfrac{T}{x},\dfrac{1}{\langle x\rangle}\right\}\right)+\\
&+O\left((\log T)^{2m}\log(3x)+\dfrac{\log T}{\log x}+\log x\left(\log T+\log(2x)\log\log(3x)\right)\right),
\end{split}
\end{align}
as long as $T\geq T_1$.

The theorem now follows by combining \eqref{firste} and \eqref{laste} and by arguing similarly as in the end of the proof of the Theorem \ref{MAIN1} for the case of $T$ being an ordinate of an $\alpha$-point.
\end{proof}

\section{Littlewood's theorem and Applications to Universality}

\begin{proof}[Proof of Theorem \ref{Main1}]
Let $n_0>1$ be the first positive integer $n$ such that $f(n)\neq0$.
We know that such a number exists, because $\mathcal{L}(s)$ differs from the only constant function in $\mathcal{S}$.
This also implies that $d_{\mathcal{L}}\geq1$.
 In the sequel we make use of the notation $a_n:=f(n)$, for all $n\in\mathbb{N}$.

Now we define the function 
$$g(s):=\left\{
\renewcommand{\arraystretch}{1.5}
\begin{array}{lllll}
\dfrac{|a_{n_0}|n_0^s}{a_{n_0}}(\mathcal{L}(s)-1),&\text{ if }\alpha=1,\\
\mathcal{L}(s)-\alpha,&\text{ if }\Re \alpha<1,\\
\alpha-\mathcal{L}(s),&\text{ if }\Re \alpha>1,\\
-i\left(\mathcal{L}(s)-\alpha\right),&\text{ if }\Re \alpha=1\text{ and }\Im \alpha>0,\\
i\left(\mathcal{L}(s)-\alpha\right),&\text{ if }\Re \alpha=1\text{ and }\Im \alpha<0,
\end{array}
\right.$$
for every $s\in\mathbb{C}\setminus\lbrace1\rbrace$ .
Then there exists $\sigma_0=\sigma_0(\alpha)>1$ such that
\begin{align}\label{gasy}
g(s)\asymp\Re(g(s))\asymp1,
\end{align}
uniformly in the half-plane $\sigma\geq\sigma_0-1$. 
Indeed, the absolute convergence of $\mathcal{L}(s)$ in the half-plane $\sigma>1$ implies that 
$$\Re(g(s))\leq|g(s)|\ll1,$$
uniformly in any half-plane $\sigma\geq\sigma_1>1$.
On the other hand, for $\sigma>1$,
$$\Re(g(s))=\left\{
\renewcommand{\arraystretch}{1.5}
\begin{array}{lllll}
|a_{n_0}|+\sum\limits_{n>n_0}\left(n_0/n\right)^{\sigma}E(n,t),&\alpha=1,\\
1-\Re \alpha+\sum\limits_{n\geq n_0}n^{-\sigma}F(n,t),&\Re \alpha<1,\\
\Re \alpha-1-\sum\limits_{n\geq n_0}n^{-\sigma}F(n,t),&\Re \alpha>1,\\
\Im \alpha+\sum\limits_{n\geq n_0}n^{-\sigma}G(n,t),&\Re \alpha=1\text{ and }\Im \alpha>0,\\
-\Im \alpha-\sum\limits_{n\geq n_0}n^{-\sigma}G(n,t),&\Re \alpha=1\text{ and }\Im \alpha<0,
\end{array}
\right.$$
where 
\begin{align*}
\renewcommand{\arraystretch}{1.5}
\begin{array}{lll}
E(n,t):=|a_{n_0}|^{-1}\left[\Re \left(\overline{a_{n_0}}a_n\right)\cos\left(t\log{{\dfrac{n_0}{n}}}\right)-\Im \left(\overline{a_{n_0}}a_n\right)\sin\left(t\log\dfrac{n_0}{n}\right)\right],\\
F(n,t):=\Re a_n\cos\left(t\log n\right)+\Im a_n\sin\left(t\log n\right),\\
G(n,t):=\Im a_n\cos\left(t\log n\right)-\Re a_n\sin\left(t\log n\right).
\end{array}
\end{align*}
Since $E(n,t),F(n,t),G(n,t)\ll_\epsilon n^{\epsilon}$ for all $n\in\mathbb{N}$ and $t\in\mathbb{R}$, there is a sufficiently large $\sigma_0>0$ such that
$$|g(s)|\geq\Re(g(z))\gg1,$$
uniformly in $\sigma\geq\sigma_0-1$.
Moreover, by arguing similarly, we can take $\sigma_0>0$ large enough such that also 
\begin{align}\label{order1}
|\mathcal{L}(s)|\gg1,
\end{align}
uniformly in $\sigma\geq\sigma_0-1$.

Lastly, we  define for every $s\in\mathbb{C}\setminus\lbrace1\rbrace$ the function $h(s):=\log g(s),$ where the logarithm takes its principal value for $\sigma\geq\sigma_0-1$, and for other points $s$ we define $h(s)$ to be the value obtained from $h(\sigma_0)$ by continuous variation along the line segments $[\sigma_0,\sigma_0+it]$ and $[\sigma_0+it,\sigma+it]$, provided that the path does not cross a zero or a pole of $g(s)$; if it does, then we take $h(s)=h(s+0)$. 
From the latter definition and \eqref{gasy} it follows that
\begin{align}\label{or}
h(s)\ll1,
\end{align}
uniformly in $\sigma\geq\sigma_0-1$.

For the second part of the proof we employ a technique used by Titchmarsh (see \cite[Theorem 9.12]{Titch2}).
It suffices to show that for every sufficiently large positive number $T$, if there is no $\alpha$-point of $\mathcal{L}(s)$ such that $|\gamma_\alpha-T|\leq\delta<1/2$, then $\delta\ll\left(\log\log\log T\right)^{-1}$.
To this end, let $T$ be a large positive number and assume that $\mathcal{L}(s)$ has no $\alpha$-point such that $|\gamma_\alpha-T|\leq\delta<1/2$. 
Then, the function $h(s)$, as it was defined above is regular in the rectangle $$\mathcal{R}:=\left\{s\in\mathbb{C}:-(\sigma_0+1)\leq\sigma\leq\sigma_0+1,\,\,\,\,\,|t-T|\leq\delta\right\}.$$
We set 
\begin{align}\label{N}
N:=\left\lfloor\dfrac{8\sigma_0}{\delta}\right\rfloor+1,
\end{align}
and we consider the circles
$$\mathcal{K}_{m,n}:=\left\{s\in\mathbb{C}:\left|s-\left(\sigma_0-\frac{n\delta}{4}+iT\right)\right|=\dfrac{m\delta}{4}\right\};$$
we define the maximum value of $h$ on them by
$$\mathbf{M}_{m,n}:=\max\limits_{s\in\mathcal{K}_{m,n}}|h(s)|$$
for every $n=0,\dots,N$ and $m=1,2,3,4$.
Observe that all circles $\mathcal{K}_{m,n}$ lie inside the strip $-(\sigma_0+1)\leq\sigma\leq\sigma_0+1$ and, in particular, in $\mathcal{R}$. 
Now the proof follows by arguing similarly as in \cite[Theorem 9.12]{Titch2}.
\end{proof}

\begin{proof}[Proof of Corollary \ref{Main2}]
Let $(\rho_{\alpha,n})_{n\in\mathbb{N}}$ be the sequence of $\alpha$-points of $\mathcal{L}(s)$.
Then, for every $k\geq k':=\lfloor\exp(3)/b\rfloor$, Theorem \ref{Main1} yields the existence of a positive integer $n_k$ such that $$\gamma_{\alpha,n_k}=bk+O\left(\dfrac{1}{\log\log\log (bk)}\right)=bk+o(1).$$
We also set $\gamma_{\alpha,n_k}=\gamma_{\alpha,n_{k'}}$ for every positive integer $k<k'$.
Now, if $a\notin b^{-1}\mathbb{Q}$ is a real number and $m$ a positive integer, then we know that the sequence $\left(a b k^m\right)_{k\in\mathbb{N}}$, is uniformly distributed mod 1  (see for example \cite[Chapter 1, Theorem 3.2]{K-N}) and $$\lim\limits_{k\to\infty}\left(a bk^m-a\gamma_{\alpha,n_{k^m}}\right)=0.$$
Hence, it follows from \cite[Chapter 1, Theorem 1.2]{K-N} that the sequence $\left(a\gamma_{\alpha,n_{k^m}}\right)_{k\in\mathbb{N}}$ is uniformly distributed mod 1.
\end{proof}

\begin{proof}[Proof of Theorem \ref{universality}]
Let $\alpha\in\mathbb{C}$ and $m=1$. 
 According to Corollary \ref{Main2}, we can find a subsequence of $\alpha$-points of $\mathcal{L}(s)$ such that
\begin{align}\label{In}
|\gamma_{\alpha, n_k}-2\pi k|<\dfrac{\pi}{2},
\end{align}
for every $k\geq k_1$, where $k_1\gg1$ is a sufficiently large and fixed positive integer. 
In addition, the sequence $\left(a\gamma_{\alpha, n_k}\right)_{k\in\mathbb{N}}$, is uniformly distributed modulo one or, equivalently from Weyl's criterion \cite[Chapter 1, Theorem 2.1]{K-N},
\begin{align}\label{W1}
\lim\limits_{N\to\infty}\dfrac{1}{N}\sum\limits_{k=1}^N\e\left(a\gamma_{\alpha, n_k}\right)=0,
\end{align}
for every $a\notin(2\pi)^{-1}\mathbb{Q}$.

Let $\mathbb{P}$ denote the set of prime numbers. Since the numbers
$$1\,\,\,\text{ and }\,\,\,\dfrac{\log p}{2\pi},\,\,\,p\in\mathbb{P},$$
are linearly independent over $\mathbb{Q}$, it follows from\cite[Chapter 1, Theorem 6.2]{K-N} and relation \eqref{W1}  that the sequence
\begin{align}\label{W2}
\left(\gamma_{\alpha, n_k}\dfrac{\log p}{2\pi}\right)_{p\in M},\,\,\,k\in\mathbb{N},
\end{align}
is uniformly distributed modulo one in $\mathbb{R}^{\sharp M}$, for any finite set $M\subseteq\mathbb{P}$. 
If we define now the truncated Euler products of $\zeta(s)$,
$$\zeta_M\left(s,\left(\theta_p\right)_{p\in M}\right):=\mathlarger\prod\limits_{p\in M}\left(1-\dfrac{\e\left(\theta_p\right)}{p^s}\right)^{-1}$$
for every $\left(s,\left(\theta_p\right)_{p\in M}\right)\in\left\{s\in\mathbb{C}:{\sigma>0}\right\}\times\mathbb{R}^{\sharp M}$ and finite $M\subseteq\mathbb{P}$, then \eqref{W2} and \cite[Lemma 4]{P} yield the existence of positive numbers $v$ and $d$ such that 
\begin{align}\label{App1}
\liminf\limits_{N\to\infty}\dfrac{1}{N}\sharp\left\{1\leq k\leq N:\max\limits_{s\in K}|\zeta_{\left\{p:p\leq y\right\}}(s+i\gamma_{\alpha, n_k},\underline{0})-f(s)|<\varepsilon\right\}>d,
\end{align}
for every $y\geq v$.

Let  $\delta:=\pi$, $T_0:=\gamma_{\alpha, n_{k_1}}-\pi/2$, $T:=\gamma_{\alpha, n_N}-\gamma_{\alpha, n_{k_1}}+\pi$, $A:=\lbrace \gamma_{\alpha, n_{k_1}},\dots,\gamma_{\alpha, n_N}\rbrace$ 
  and 
 \begin{align*}
 f(t):=\zeta(s+it)-\zeta_{\lbrace p:p\leq y\rbrace}(s+it,\underline{0}).
 \end{align*}
Then inequality \eqref{In} and Lemma \ref{gall} yield
 \begin{align}\label{G}
 \begin{split}
\sum\limits_{k=k_1}^N\hspace{-1pt}|\zeta(s+i\gamma_{\alpha, n_k})\hspace{-1pt}-\hspace{-1pt}\zeta_{\lbrace p:p\leq y\rbrace}(s+i\gamma_{\alpha, n_k},\underline{0})|^2\leq&\,2\int\limits_{0}^{2\pi(N+1)}|f(x)|^2\mathrm{d}x+\\
&+\hspace{-1pt}\left(\int\limits_{0}^{2\pi(N+1)}|f(x)|^2\mathrm{d}x\hspace{-1pt}\int\limits_{0}^{2\pi(N+1)}\left|f'(x)\right|^2\mathrm{d}x\right)^{1/2}.
\end{split}
\end{align}
Bohr \cite[Hilfssatz 2]{Bohr} proved that
\begin{align}\label{Bo}
\lim\limits_{y\to\infty}\limsup\limits_{T\to\infty}\dfrac{1}{T}\int\limits_{0}^{T}|\zeta(s+i\tau)-\zeta_{\left\{p:p\leq y\right\}}(s+i\tau,\underline{0})|^2\mathrm{d}\tau=0,
\end{align}
uniformly on compact subsets of $\mathcal{D}$. It is clear that an application of Cauchy's integral formula would yield \eqref{Bo} also in the case of $\zeta'-\zeta'_{\left\{p:p\leq y\right\}}$. 
Therefore, in view of  \eqref{G} and \eqref{Bo} it follows that
\begin{align}\label{A}
\lim\limits_{y\to\infty}\limsup\limits_{N\to\infty}\dfrac{1}{N}\sum\limits_{k=1}^{N}|\zeta(s+i\gamma_{\alpha, n_k})-\zeta_{\lbrace p:p\leq y\rbrace}(s+i\gamma_{\alpha, n_k},\underline{0})|^2=0,
\end{align}
uniformly on compact subsets of $\mathcal{D}$. 

If $G\subseteq\mathbb{C}$ is a bounded domain and $g$ an analytic and square integrable  function on $G$, then
$$|g(z)|\leq\left(\sqrt{\pi}d(z,\partial G)\right)^{-1}\left(\mathlarger\iint_G|g(\sigma+it)|^2\mathrm{d}\sigma\mathrm{d}t\right)^{1/2},$$
for every $z\in G$ (see for example \cite[Chapter 1, Theorem 1]{D-S}). Thus, by considering a bounded domain $G$ such that $K\subseteq G\subseteq\overline{G}\subseteq\mathcal{D}$, relation \eqref{A} and Chebyshev's inequality yield the existence of a positive number $u$ such that
\begin{align}\label{App2}
\liminf\limits_{N\to\infty}\dfrac{1}{N}\sharp\left\{1\leq k\leq N:\max\limits_{s\in K}\left|\zeta\left(s+i\gamma_{\alpha, n_k}\right)-\zeta_{\lbrace p: p\leq y\rbrace}\left(s+i\gamma_{\alpha, n_k},\underline{0}\right)\right|<\tilde{\varepsilon}
\right\}>1-\tilde{\varepsilon},
\end{align}
for every $y\geq u$, where $\tilde{\varepsilon}=\min\lbrace\varepsilon,d/2\rbrace$.

Now the theorem follows from \eqref{App1} and \eqref{App2} by taking $y=\max\lbrace u,v\rbrace$.
\end{proof}

\newpage
\small
Athanasios Sourmelidis,\\
Institute of Analysis and Number Theory, TU Graz\\
Steyrergasse 30, 8010 Graz, Austria\\
 sourmelidis@math.tugraz.at
\medskip
\medskip

Teerapat Srichan\\
Department of Mathematics, Faculty of Science,\\
Kasetsart University, Bangkok 10900, Thailand \\
fscitrp@ku.ac.th
\medskip
\medskip

 J\"orn Steuding\\
Institute for Mathematics, W\" urzburg University, \\
Emil-Fischer Str. 40, 97074 W\"urzburg, Germany\\
 athanasios.sourmelidis@mathematik.uni-wuerzburg.de\\
steuding@mathematik.uni-wuerzburg.de

\end{document}